\documentclass[11pt]{amsart}
\theoremstyle{plain}
\newtheorem{thm}{Theorem}[section]
\newtheorem{theorem}[thm]{Theorem}

\newtheorem{lemma}[thm]{Lemma}
\newtheorem{corollary}[thm]{Corollary}

\theoremstyle{definition}

\newtheorem{remarks}[thm]{Remarks}

\numberwithin{equation}{section}

 \title[Dynamical degrees, Weil's and the standard conjectures]{Relations between dynamical degrees, Weil's Riemann hypothesis and the standard conjectures}
   \author{Tuyen Trung Truong}
 \address{School of Mathematical Sciences, The University of Adelaide, Adelaide SA 5005, Australia}
 \email{tuyen.truong@adelaide.edu.au}

\thanks{The author was supported by Australian Research Council grants DP120104110 and DP150103442. }
    \date{\today}

\begin{document}
\maketitle
\begin{abstract}
Let $\mathbb{K}$ be an algebraically closed field, $X$ a smooth projective variety over $\mathbb{K}$ and $f:X\rightarrow X$ a dominant regular morphism. Let $N^i(X)$ be the group of algebraic cycles modulo numerical equivalence. Let $\chi (f)$ be the spectral radius of the pullback $f^*:H^*(X,\mathbb{Q}_l)\rightarrow H^*(X,\mathbb{Q}_l)$ on $l$-adic cohomology groups, and $\lambda (f)$ the spectral radius of the pullback $f^*:N^*(X)\rightarrow N^*(X)$. We prove in this paper, by using consequences of Deligne's proof of Weil's Riemann hypothesis, that $\chi (f)=\lambda (f)$. This answers affirmatively a question posed by Esnault and Srinivas. Consequently, the algebraic entropy $\log \chi (f)$ of an endomorphism is both a birational invariant and \'etale invariant.  More general results are proven if  either $\mathbb{K}=\overline{\mathbb{F}_p}$ or the Fundamental Conjecture D (numerical equivalence vs homological equivalence) holds.  Among other results in the paper, we show that if some properties of dynamical degrees, known in the case $\mathbb{K}=\mathbb{C}$, hold in positive characteristics, then simple proofs of Weil's Riemann hypothesis follow.

\end{abstract}

\section{Introduction}
\label{SectionIntroduction}
The proof of the Weil's conjecture by Deligne is one of the major achievements of mathematics in the 20th century. Through the visions of the likes of Weil and Grothendieck, the question about counting the (asymptotic) number of points in finite fields $\mathbb{F}_{q^n}$, as $n\rightarrow \infty$, on a smooth projective variety $X_0$ defined on $\mathbb{F}_q$, is translated to the question about the eigenvalues of the pullbacks $(Fr^n)^*$ on \'etale cohomology groups $H^{*}(X,\mathbb{Q}_l)$. Here $X$ is the lift of $X_0$ to an algebraic closure $\overline{\mathbb{F}}_q$ of $\mathbb{F}_q$, and $Fr$ is the Frobenius map. Bombieri and Grothendieck thought of solving the Weil's Riemann hypothesis via the famous standard conjectures, but the actual proof by Deligne \cite{deligne1, deligne2} was totally different and surprising. For some good references about this, see for example \cite{milne, milne1}. 

The current  paper serves two purposes. First, we use the Weil's Riemann hypothesis or the Fundamental Conjecture D (that modulo torsions numerical and homological equivalences coincide  on algebraic cycles) to extend several known results on dynamical degrees from complex dynamics to positive characteristics.  Conversely, the second purpose is to point out that if some stronger results on dynamical degrees, which again hold for $\mathbb{K}=\mathbb{C}$, hold in positive characteristic, then new simple proofs of the Weil's Riemann hypothesis follow. Thus, it is demonstrated here that there is a curious relation between algebraic dynamics and Weil's cohomology theories. 

This paper was inspired by the results of Esnault and Srinivas \cite{esnault-srinivas} on automorphisms of surfaces. In the remaining of this introduction, we pose some questions to be studied in the paper and then state the main results. To make the presentation concise, we collect some background materials on correspondences and dynamical degrees in Section \ref{Section1}. 

\subsection{Questions}

Let $\mathbb{K}$ be an algebraically closed field, $X$ a smooth projective variety over $\mathbb{K}$ and $f:X\rightarrow X$ a correspondence. (A correspondence is roughly an algebraic cycle of $X\times X$ whose dimension is exact $\dim (X)$.  Examples of interest include regular morphisms, rational maps and a linear combination of such. See Section \ref{Section1} for a precise definition.) The Weil's Riemann hypothesis can be stated in terms of the following numbers $\lambda _i$ and $\chi _i$, see Theorem \ref{Theorem4} and Section \ref{Section3} for more details.

We first consider the groups $N^i(X)$ of algebraic cycles of codimension $i$ modulo numerical equivalence. These are free Abelian groups of finite ranks (see Chapter 19 in \cite{fulton} or Section 6.2 in \cite{esnault-srinivas}). We define $\lambda _i(f)$ to be the numbers
\begin{eqnarray*}
\lambda _i(f):=\limsup _{n\rightarrow\infty} \|(f^n)^*|_{N^{i}(X)}\|^{1/n},
\end{eqnarray*}
where we fix any norm on the finite-dimensional vector space $N^{i}(X)\otimes \mathbb{R}$. In \cite{truong9} (recalled in (\ref{Equation1}) in Section \ref{Section1} below), we showed that in fact the limsup can be replaced by lim, and all numbers $\lambda _i(f)$ are finite.  

Also, we define $\chi _i(f)$ to be the numbers
\begin{eqnarray*}
\chi _i(f):=\limsup _{n\rightarrow\infty} \|(f^n)^*|_{H^{i}(X,\mathbb{Q}_l)}\|^{1/n},
\end{eqnarray*}
here we fix any norm on the finite-dimensional vector space $H^{i}(X,\mathbb{Q}_l)$. In contrast to the $\lambda _i(f)$'s, the finiteness of $\chi _i(f)$'s are not obvious, although by definition we have $\chi _{2i}(f)\geq \lambda _i(f)$. We also do not know whether the limsup in the definition for $\chi _i(f)$ can be replaced by lim.  

We may call the number
\begin{eqnarray*}
\log \chi (f):=\log \max _{i=0,\ldots ,2\dim (X)}\chi _i(f),
\end{eqnarray*}
the algebraic entropy. 

These numbers $\lambda _{i}(f)$ and $\chi _{i}(f)$ have been extensively studied when $\mathbb{K}=\mathbb{C}$ in the context of complex dynamics. They are called dynamical degrees in that setting and are important to the dynamical properties of $f$, see Section \ref{Section1} for more details. The known results in the case $\mathbb{K}=\mathbb{C}$ (see Section \ref{Section1}) and recent results of Esnault and Srinivas \cite{esnault-srinivas} on automorphisms of surfaces over positive characteristic inspire us to study the following questions. 

{\bf Question 1.} Is $\chi _i(f)$ finite, for all $i=0,\ldots ,2\dim (X)$? 

{\bf Question 2.} Is $\chi _{2i}(f)=\lambda _i(f)$ for all $i$?

{\bf Question 3.} (Product formula) Let $f,g:X\rightarrow Y$ be dominant rational maps and $\pi :X\rightarrow Y$ be a dominant rational map so that $\pi\circ f=g\circ \pi$. Is it true that we can define the relative dynamical degrees $\chi _{2i}(f|\pi )$ which satisfy the relations
\begin{eqnarray*}
\chi _{2p}(f)=\max _{0\leq i\leq \dim (Y),~0\leq p-i\leq \dim (X)-\dim (Y)}\chi _{2i}(g)\chi _{2(p-i)}(f|\pi ),
\end{eqnarray*}
for all $p=0,\ldots ,\dim (X)$?

{\bf Question 4.} (Dinh's inequality) Is $\chi _{i}(f)^2\leq \max _{p+q=i}\lambda _p(f)\lambda _q(f)$? 

The following weaker version of Question 4 is enough for applications to dynamics (c.f. Gromov - Yomdin's theorem and Gromov - Dinh - Sibony's inequality, see Section \ref{Section1}):

{\bf Question 4'.} Is $\max _{i=0,\ldots ,2\dim (X)}\chi _i(f)=\max _{i=0,\ldots ,\dim (X)}\lambda _i(f)$?

{\bf Question 5.} Are $\chi _{2i}(f)$ birational invariants? 

Note that the answers to all of these questions are affirmative when $\mathbb{K}=\mathbb{C}$, see Section \ref{Section1}. A crucial advantage in working with $\mathbb{C}$ is that we have  positivity notions, induced from positive closed forms and currents, on cohomology classes. These positivity notions are not yet available on fields of positive characteristics.  

\subsection{Main results} Here we state main results of the paper. We recall that we work on an algebraically closed field $\mathbb{K}$ of arbitrary characteristic. We use the convention that a variety is irreducible. 

We mention a relevant fundamental conjecture on algebraic cycles. We denote by $\mathcal{Z}^i_{hom}(X)$ the set of algebraic cycles on $X$ of codimension $i$ whose image in $H^{2i}(X,\mathbb{Q}_l)$ is $0$; and by $\mathcal{Z}^i_{num}(X)$ the set of algebraic cycles on $X$ of codimension $i$ which are $0$ under the numerical equivalence, that is those cycles $V$ for which $V.W=0$ for all algebraic cycles $W$ of dimension $i$. The following weaker version of the Fundamental Conjecture D is sufficient for our purpose. 

{\bf The numerical - homological equivalences condition.} Given a smooth projective variety $Z$ of {\bf even} dimension $2k'$, we say that $NH(Z)$ holds if $\mathcal{Z}^{k'}_{hom}(Z)\otimes \mathbb{Q}=\mathcal{Z}^{k'}_{num}(Z)\otimes \mathbb{Q}$ for the middle-degree cohomology group $H^{2k'}(Z,\mathbb{Q}_l)$.  

The first result answers Questions 1 and 4'.
\begin{theorem}
1) Assume that  $NH(X\times X)$ holds. Then, Questions 1 and 4' have affirmative answers. More precisely, if $f:X\rightarrow X$ is a dominant correspondence, then
\begin{eqnarray*}
\chi _i(f)\leq \max _{p=0,\ldots ,\dim (X)}\lambda _p(f),
\end{eqnarray*}
for all $i=0,\ldots ,2\dim (X)$.

2) Assume that $f:X\rightarrow X$ is a dominant {\bf regular morphism}. Then Questions 1 and 4' have affirmative answers.  
\label{Theorem1}\end{theorem}
Part 2) of the theorem answers affirmatively a question posed in \cite{esnault-srinivas} (see Section 6.3 therein).  In the case $X$ is defined on a finite field, part 1) of Theorem \ref{Theorem1} holds unconditionally as well. 
\begin{theorem}
Let $\mathbb{K}=\overline{\mathbb{F}_p}$, the algebraic closure of a finite field $\mathbb{F}_p$. Let $X$ be a smooth projective variety over $\mathbb{K}$, and $f:X\rightarrow X$ a correspondence. Then,
\begin{eqnarray*}
\chi _i(f)\leq \max _{p=0,\ldots ,\dim (X)}\lambda _p(f),
\end{eqnarray*}
for all $i=0,\ldots ,2\dim (X)$.
\label{Theorem6}\end{theorem}

As a consequence of Theorems \ref{Theorem1} and \ref{Theorem6}, the algebraic entropy $\log \chi (f)$ of a dominant endomorphism is both a birational invariant and \'etale invariant. More generally, we have the following result, which is an analog of a classical result of Bowen on topological entropy of continuous dynamical systems on compact metric spaces (Theorem 17 in \cite{bowen}). By the proof of the consequence,  provided $\mathbb{K}=\overline{F}_p$ or the Fundamental Conjecture D holds, the same conclusion holds for rational maps and a slightly weaker conclusion holds for all correspondences. 
\begin{corollary}
Let $X,Y$ be smooth projective varieties over $\mathbb{K}$ of the same dimension,  $f:X\rightarrow X$ and $g:Y\rightarrow Y$ dominant {\bf regular morphisms}. Assume that there is a dominant rational map (necessarily has generic finite fibres) $\pi :X\rightarrow Y$ so that $\pi \circ f=g\circ \pi$. Then 
\begin{eqnarray*}
\max _{i=0,\ldots ,2\dim (X)}\chi _i(f)=\max _{i=0,\ldots ,2\dim (X)}\chi _i(g).
\end{eqnarray*}
\label{Corollary1}\end{corollary}
\begin{proof}
This follows from the corresponding properties for the geometric dynamical degrees $\lambda _i(f)$ and $\lambda _i(g)$ (Theorems 1.1 and 1.3 in \cite{truong9}) and part 2 of Theorem \ref{Theorem1} above.  
\end{proof}

 As another consequence of Theorems \ref{Theorem1} and \ref{Theorem6}, we answer Question 2 for a large class of correspondences on surfaces. 
\begin{theorem}
1) Let $X$ be a smooth projective surface over $\mathbb{K}$, and $f:X\rightarrow X$ a dominant correspondence with $\lambda _1(f)\geq \max \{\lambda _0(f),\lambda _2(f)\}$. Assume that $NH(X\times X)$ holds. Then $\chi _2(f)=\lambda _1(f)$. Moreover, $\max \{\chi _1(f),\chi _3(f)\}\leq \lambda _1(f)$. 

1') Let $X$ be a smooth projective surface over $\mathbb{K}$, and $f:X\rightarrow X$ a dominant correspondence with $\lambda _1(f)\geq \max \{\lambda _0(f),\lambda _2(f)\}$. Assume that $\mathbb{K}=\overline{\mathbb{F}_p}$ is the closure of a finite field $\mathbb{F}_p$. Then $\chi _2(f)=\lambda _1(f)$. Moreover, $\max \{\chi _1(f),\chi _3(f)\}\leq \lambda _1(f)$. 

2) Let $X$ be a smooth projective surface over $\mathbb{K}$, and $f:X\rightarrow X$ a dominant {\bf regular morphism} with $\lambda _1(f)\geq \max \{\lambda _0(f),\lambda _2(f)\}$. Then $\chi _2(f)=\lambda _1(f)$. Moreover, $\max \{\chi _1(f),\chi _3(f)\}\leq \lambda _1(f)$. 
\label{Theorem2}\end{theorem}
If $f$ is an automorphism (or more generally, a birational map) then $\lambda _0(f)=\lambda _2(f)=1$ and $\lambda _1(f)\geq 1$. Hence, parts 1') and  2) of our theorem applies. Note that this case, i.e. $f$ is an automorphism of a surface, was solved by Esnault and Srinivas in \cite{esnault-srinivas}. Their proof makes use of the classification of surfaces and is not purely algebraic (because at some part of the proof, they need to use the lifting to characteristic 0, and use the known results in that case). They also mentioned an algebraic proof of their result, suggested by P. Deligne, under the assumption that the standard conjectures hold.  

{\bf Remark.} By the results in \cite{katz-messing}, all the above results are valid for any Weil's cohomology theory. 

Some other results related to Questions 2 and 3 will be proven in the last section of this paper. The last main result concerns the relation between the above questions and the Weil's Riemann hypothesis. 
\begin{theorem}
If Question 2 or Question 3 or Question 4 has an affirmative answer then we obtain a simple new proof of the Weil's Riemann hypothesis. 
\label{Theorem4}\end{theorem}

\subsection{Plan of the paper} Some background materials are collected in Section \ref{Section1}. In Section \ref{Section2} we prove Theorems \ref{Theorem1}, \ref{Theorem6} and \ref{Theorem2}. In Section \ref{Section3} we prove Theorem \ref{Theorem4}. In the last section we discuss an approach toward solving Questions 2 and 3, the main result in that section is Theorem \ref{Theorem3}.

Two main ideas are used throughout the paper. The first one is that by working on $X\times X$, some questions about pulling back of a correspondence $f:X\rightarrow X$ on $l$-adic cohomology groups may be reduced to questions about algebraic cycles only. The second one is that given a dominant correspondence $f:X\rightarrow X$ and a dominant regular morphism with {\bf finite} fibres $\pi :X\rightarrow Y$, we can consider the pushforward $g_n=\pi _*(f^n):Y\rightarrow Y$ to study the dynamics of $f$. 

{\bf Acknowledgments.} The author was very much benefited from the invaluable and generous help of and inspiring discussions with Peter O'Sullivan, to whom he gratefully expresses his thankfulness. We are indebted to H\'el\`ene Esnault and Keiji Oguiso for their interest in the paper and important corrections, to them and Tien-Cuong Dinh for helpful comments and suggestions, which greatly improved the presentation of the paper. In particular, H\'el\`ene Esnault's questions and information to us helped to clarify many points in the proofs of the results.  Part of Theorem \ref{Theorem4}, on the relation between Question 2 and the Weil's Riemann hypothesis, was presented in the author's talk at the conference "Geometry at the ANU, August 2016". He would like to thank the organisers of the conference for the invitation and hospitality. 

\section{Preliminaries}
\label{Section1}

In this section, we recall some backgrounds on correspondences and dynamical degrees. 

\subsection{A brief summary on correspondences}
Let $\mathbb{K}$ be a field and $X,Y$ irreducible (not necessarily smooth or projective) varieties. A {correspondence}  $f:X\rightarrow Y$ is given by an algebraic cycle $\Gamma _f=\sum _{i=1}^m\Gamma _i$ on $X\times Y$, where $m$ is a positive integer and $\Gamma _i\subset X\times Y$ are irreducible subvarieties of dimension exactly $\dim (X)$. We do not assume that $\Gamma _i$ are distinct, and hence may write the above sum as $\sum _{j}a_j\Gamma _j$ where $\Gamma _j$ are distinct and $a_j$ are positive integers. We will call $\Gamma _f$ the graph of $f$, by abusing the usual notation when $f$ is a rational map.  If $f$ is a correspondence and $a\in \mathbb{N}$, we denote by $af$ the correspondence whose graph is $a\Gamma _f$. In other words, if $\Gamma _f=\sum _{i}\Gamma _i$ then $\Gamma _{af}=\sum _{i}a\Gamma _i$.  If $\Gamma _f=a\Gamma$ where $\Gamma$ is irreducible and $a\in \mathbb{N}$, we say that the correspondence $f$ is {irreducible}. A  rational map $f$ is an irreducible correspondence, since its graph is irreducible. A correspondence is dominant if for each $i$ in the sum, the two natural projections from $\Gamma _i$ to $X,Y$ are dominant. Dominant correspondences can be composed and the resulting correspondence is also dominant. Given two dominant correspondences $f:X\rightarrow X$ and $g:Y\rightarrow Y$, we say that they are semi-conjugate if there is a dominant rational {\bf map} $\pi :X\rightarrow Y$ such that $\pi \circ f=g\circ \pi$. We will simply write $\pi :(X,f)\rightarrow (Y,g)$ to mean that $\pi$ is a dominant rational map semi-conjugating $(X,f)$ and $(Y,g)$. 

Let $\pi :X\rightarrow Y$ be a dominant {\bf regular morphism} with {\bf finite} fibres,  $f:X\rightarrow X$ and $g:Y\rightarrow Y$ dominant correspondences. We define $\pi ^*(g)$ to be the correspondence on $X$ whose graph is $(\pi \times \pi )^{*}(\Gamma _g)$, and define $\pi _*(f)$ to be the correspondence on $Y$ whose graph is $(\pi \times \pi )_*(\Gamma _f).$

{\bf Remarks.} If $f:X\rightarrow X$ is a correspondence, then we can define pullback and pushforward of algebraic cycles and cohomology classes in the following way. Let $pr_1,pr_2:X\times X\rightarrow X$ be the projections. Then 
\begin{eqnarray*}
f^*(\alpha )&:=&(pr _1)_*[pr_2^*(\alpha ).\Gamma _f],\\
f_*(\alpha )&:=&(pr _2)_*[pr_1^*(\alpha ).\Gamma _f].
\end{eqnarray*}
Note that (in contrast to a more common use of correspondences in Algebraic Geometry), the definition of compositions of dominant correspondences in this paper is modelled after that of the compositions of rational maps. Therefore, in general we have $(f^2)^*\not= (f^*)^2$, and so on. This phenomenon of non-compatibility between pullback and iterate was first studied on projective spaces in \cite{fornaess-sibony}, under the name of algebraic instability. One simple example is that of the standard Cremona map $f:\mathbb{P}^2\rightarrow \mathbb{P}^2$ given by the formula: $f[x_0,x_1,x_2]=[x_1x_2:x_2x_0:x_0x_1]$.

\subsection{Relative dynamical degrees on complex projective varieties} One of the main advantages utilised when working in dynamics over the complex field $\mathbb{C}$ is the existence of positive closed forms and currents, and consequently a positivity notion for cohomological classes. 

One important tool in Complex Dynamics is dynamical degrees for dominant meromorphic selfmaps. They are bimeromorphic invariants of a meromorphic selfmap $f:X\rightarrow X$ of a compact K\"ahler manifold $X$. The $p$-th dynamical degree $\lambda _p(f)$  is the exponential growth rate of the spectral radii of the pullbacks $(f^n)^*$ on the Dolbeault  cohomology group $H^{p,p}(X)$. For a surjective holomorphic map $f$, the dynamical degree $\lambda _p(f)$ is simply the spectral radius of $f^*:H^{p,p}(X)\rightarrow H^{p,p}(X)$.  Recall that for a linear map $L$ on a complex vector space, the spectral radius $sp(L)$ is the maximum of the absolute values of eigenvalues of $L$. Fundamental results of Gromov \cite{gromov} and Yomdin \cite{yomdin} expressed the topological entropy of a surjective holomorphic map in terms of its dynamical degrees: $$h_{top}(f)=\log \max _{0\leq p\leq {\dim (X)}}\lambda _{p}(f).$$ Since then, dynamical degrees have played a more and  more important role in dynamics of meromorphic maps. In many results and conjectures in Complex Dynamics in higher dimensions, dynamical degrees play a central role. 

Let $X$ be a compact K\"ahler manifold of dimension $k$ with a K\"ahler form $\omega _X$, and let $f:X\rightarrow X$ be a dominant meromorphic map. For $0\leq p\leq k$, the $p$-th dynamical degree $\lambda _p(f)$ of $f$ is defined as follows 
\begin{equation} 
\lambda _p(f)=\lim _{n\rightarrow\infty}(\int _X(f^n)^*(\omega _X^p)\wedge \omega _X^{k-p})^{1/n}=\lim _{n\rightarrow\infty}r_p(f^n)^{1/n},
\label{Equation01}\end{equation}
where $r_p(f^n)$ is the spectral radius of the linear map $(f^n)^*:H^{p,p}(X)\rightarrow H^{p,p}(X)$. The existence of the limit in the above expression is non-trivial and has been proven by Russakovskii and Shiffman \cite{russakovskii-shiffman} when $X=\mathbb{P}^k$, and by Dinh and Sibony \cite{dinh-sibony10, dinh-sibony1} when $X$ is compact K\"ahler. Both of these results use regularisation of positive closed currents. The limit in (\ref{Equation01}) is important in showing that dynamical degrees are birational invariants. The dynamical degrees satisfy the log-concavity: $\lambda _i(f)\lambda _{i+2}(f)\leq \lambda _{i+1}(f)^2$ for all $i=0,\ldots ,\dim (X)$. This is a consequence of the mixed Hodge-Riemann theorem. Remark: The first dynamical degree $\lambda _1(f)$ was used earlier to study Green currents in complex dynamics (first introduced by N. Sibony), see e.g. \cite{bedford-smillie} for surfaces and \cite{fornaess-sibony} for higher dimensions. 

For meromorphic maps of compact K\"ahler manifolds with invariant fibrations, a more general notion called relative dynamical degrees has been defined by Dinh and Nguyen in \cite{dinh-nguyen}. (Here, by a fibration we simply mean a dominant rational map, without any additional requirements.) The "product formulas" (see below) provide a very useful tool to check whether a meromorphic map is primitive (i.e.  has no invariant fibrations over a base which is of smaller dimension and not a point, see \cite{zhang}). In another direction,  when $\mathbb{K}=\mathbb{C}$,  Dinh and Sibony \cite{dinh-sibony11} defined dynamical degrees and topological entropy for meromorphic correspondences over irreducible varieties. For any dominant correspondence $f$, the following Gromov - Dinh - Sibony's inequality holds:
$$h_{top}(f)\leq \log \max _{0\leq p\leq {\dim (X)}}\lambda _{p}(f).$$

Computations of dynamical degrees of so-called Hurwitz correspondences of the moduli spaces $\mathcal{M}_{0,N}$  were given in \cite{ramadas}, wherein a proof that dynamical degrees of correspondences (over $\mathbb{K}=\mathbb{C}$, and for irreducible varieties) are birational invariants was also given. 

When $\mathbb{K}=\mathbb{C}$, we can use the fact that the cohomological class of a very ample divisor on $X$ represents a K\"ahler form, to deduce that the dynamical degrees defined above can also be computed in terms of algebraic cycles on $X$. 

\subsubsection{Product formula}
Let $f:X\rightarrow X$ and $g:Y\rightarrow Y$ be dominant rational maps, where $X$ and $Y$ are smooth projective varieties. Assume also that there is a dominant rational map $\pi :X\rightarrow Y$ so that $\pi \circ f=g\circ \pi$. Dinh and Nguyen \cite{dinh-nguyen} defined relative dynamical degrees $\lambda _i(f|\pi )$ for $i=0,\ldots ,\dim (X)-\dim (Y)$, which are  birational invariants. In case $Y=$ a point, these relative dynamical degrees are the same as the dynamical degrees mentioned above. Moreover, they also defined relative dynamical degrees in the K\"ahler setting. They proved the following result in the algebraic setting: 

{\bf Product formula.} For all $p=0,\ldots ,\dim (X)$, we have 
\begin{eqnarray*}
\lambda _p(f)=\max _{0\leq i\leq \dim (Y),~0\leq p-i\leq \dim (X)-\dim (Y)}\lambda _i(g)\lambda _{p-i}(f|\pi ).
\end{eqnarray*} 
 
The product formula in the K\"ahler setting for meromorphic maps was proven in \cite{dinh-nguyen-truong1}. We have three special cases. 

{\bf Case 1.} $X=Y\times Z$, $f=g\times h$ is a product map, and $\pi :X=Y\times Z\rightarrow Y$ is the projection onto $Y$. In this case $\lambda _j(f|\pi )=\lambda _j(h)$ for all $j$. Proving the product formula in this case is, via the Kunneth's formula, reduced to simple properties of the eigenvalues of a tensor product of linear maps. 

{\bf Case 2.} Assume that $y^0\in Y$ is a "good" periodic point of order $m$ of $g$. Then $\lambda _j(f|\pi )=\lambda _j(f^m|_{\pi ^{-1}(y^0)})^{1/m}$. This explains the use of the notation and also the intuitive meaning that relative dynamical degrees are the dynamical degrees of the restriction of $f$ on the fibres of $\pi$.

{\bf Case 3.} $\dim (X)=\dim (Y)$ (equi-dimensional). In this case $\lambda _j(f)=\lambda _j(g)$ for all $j$, and the only relative dynamical degree is $\lambda _0(f|\pi )=1$.  

\subsubsection{Dinh's inequality}
We define:
\begin{eqnarray*}
\chi _i(f):=\limsup _{n\rightarrow\infty}\| (f^n)^*|_{H^i(X,\mathbb{C})}\|^{1/n}.
\end{eqnarray*}
Here, we can choose any norm on the finite-dimensional vector space $H^i(X,\mathbb{C})$. From the results mentioned above, $\chi _{2i}(f)=\lambda _i(f)$ and in this case the limsup can be replaced by lim. However, when $i$ is an odd number,  we do not know whether the limsup in the definition can be replaced by lim.  Dinh \cite{dinh} showed the following inequality, by using weakly positive closed smooth forms: 
\begin{eqnarray*}
\chi _i(f)^2\leq \max _{p+q=i}\lambda _p(f)\lambda _q(f). 
\end{eqnarray*}
 
\subsection{Relative dynamical degrees in positive characteristics} One main difficulty in extending the results in the previous section to positive characteristic is that there is not yet a positivity notion on $l$-adic cohomology groups. 

Recently, work on birational maps of surfaces over an algebraically closed field of arbitrary characteristic has become more and more popular. As some examples, we refer the readers  to \cite{esnault-srinivas, xie, blanc-cantat, esnault-oguiso-yu, oguiso}. In these results, (relative) dynamical degrees also play an important role. 

In the case of positive characteristic, positivity notions are not yet available on $l$-adic  cohomology groups. This lets open the question of how to define cohomological dynamical degrees in the case of positive characteristic. In contrast, in \cite{truong9} (part 1 in Theorem 1.1 there)  the author was able to show that the following limits 
\begin{equation}
\lambda _i(f):=\lim _{n\rightarrow\infty}((f^n)^*(H^i).H^{\dim (X)-i})^{1/n},
\label{Equation1}\end{equation}
exist, for all $i=0,\ldots ,\dim (X)$, over an arbitrary field. Here $f$ is a correspondence and $H$ is any very ample divisor on $X$. The definition can also be adapted to the case where $X$ is singular or not irreducible, by using de Jong's alterations and pullbacks of correspondences by equi-dimensional dominant rational maps. Hence it is justified to call these the geometric dynamical degrees. These geometric dynamical degrees are again birational invariants (parts 2 of Theorem 1.1  in \cite{truong9} ). The product formula is also proven in part 4) of Theorem 1.1 in \cite{truong9} in the setting of correspondences.  For some possible applications of these to  topological entropy, in particular the Gromov - Yomdin's theorem, see \cite{truong11}. (After sending out an earlier version of \cite{truong9}, we were informed by Charles Favre that Nguyen-Bac Dang had been developing an alternative approach for (relative) geometric dynamical degrees of rational maps on normal projective varieties.)  

For a {\bf regular morphism} $f$, we declare $\chi _i(f):=$ the spectral radius of the linear map $f^*:H^{i}(X,\mathbb{Q}_l)\rightarrow H^{i}(X,\mathbb{Q}_l)$. Here we use any fixed embedding of $\mathbb{Q}_l$ into $\mathbb{C}$.  (Remark: As a consequence of the Riemann hypothesis for positive characteristic, which was the last and crucial part of the Weil's conjectures and proven by Deligne, this $\chi _i(f)$ is independent of the embedding of $\mathbb{Q}_l$. However, we won't assume this in the below.) Since the $l$-adic cohomology groups are still not well-understood, even computing the $\chi _i(f)$ on surfaces is quite a challenging task in practice. In contrast, as mentioned above, the geometric dynamical degrees $\lambda _i(f)$ have some good functorial properties which make computations easier. For example (see Section \ref{Section3} for more details), computing the geometric dynamical degrees of the Frobenius map on any smooth projective variety $X$ can be done by applying the product formula to a dominant regular morphism $\pi :X\rightarrow \mathbb{P}^k$ with finite fibres, utilising the fact that the dynamical degrees of a regular morphism of $\mathbb{P}^k$ are very easy to describe.     

For a general correspondence, taking the clues from the case $\mathbb{K}=\mathbb{C}$, we may proceed as follows. Define $\chi _i(f)$ to be the numbers
\begin{eqnarray*}
\chi _i(f):=\limsup _{n\rightarrow\infty} \|(f^n)^*|_{H^{i}(X,\mathbb{Q}_l)}\|^{1/n},
\end{eqnarray*}
here we fix any norm on the finite-dimensional vector space $H^{i}(X,\mathbb{Q}_l)$. We may call the number
\begin{eqnarray*}
\log \chi (f):=\log \max _{i=0,\ldots ,2\dim (X)}\chi _i(f),
\end{eqnarray*}
the algebraic entropy. 

Note that we always have $\chi _{2i}(f)\geq \lambda _i(f)$, but the finiteness of the above numbers $\chi _i(f)$ is not obvious. We expect that known results for relative dynamical degrees on $\mathbb{K}=\mathbb{C}$ should be carried out to an arbitrary field. 

 \section{Proofs of Theorems \ref{Theorem1}, \ref{Theorem6} and \ref{Theorem2}} 
\label{Section2}

{\bf Convention.} Strictly speaking, for the arguments below to be extremely rigorous, we need to use the Tate twists of the l-adic cohomology groups in various places. For example, a subvariety of codimension $c$ of $X$ has cohomology class in $H^{2c}(X,\mathbb{Q}_l(c))$. Similarly, we also need to use a twist in the Poincar\'e duality. However, since $H^{2c}(X,\mathbb{Q}_l(c))=H^{2c}(X,\mathbb{Q}_l)\otimes \mathbb{Q}_l(c)$, and $\mathbb{Q}_l(c)$ is a $1$-dimensional $\mathbb{Q}_l$ vector space, the computations and estimates on $H^{2c}(X,\mathbb{Q}_l(c))$ and $H^{2c}(X,\mathbb{Q}_l)$ are almost identical. For simplicity, the symbols for the twists are suppressed. (See also Remark 25.5 in \cite{milne}.)

Let $Z$ be a smooth projective variety of {\bf even} dimension $2k'$. Assume that $NH(Z)$ holds. We then construct a useful decomposition on $H^{2k'}(Z,\mathbb{Q}_l)$. By Poincar\'e duality, the intersection product
\begin{eqnarray*}
(,):~H^{2k'}(Z,\mathbb{Q}_l)\times H^{2k'}(Z,\mathbb{Q}_l) \rightarrow \mathbb{Q}_l,
\end{eqnarray*}
is symmetric and non-degenerate. Under the assumption that $NH(Z)$ holds, we will prove that there is a decomposition: 
\begin{eqnarray*}
H^{2k'}(Z,\mathbb{Q}_l)=H^{2k'}_{alg}(Z,\mathbb{Q}_l)\oplus H^{2k'}_{tr}(Z,\mathbb{Q}_l).
\end{eqnarray*}
Here $H^{2k'}_{alg}(Z,\mathbb{Q}_l)$ (the algebraic part) is the subvector space generated by the images of algebraic cycles in $H^{2k'}(Z,\mathbb{Q}_l)$; and $H^{2k'}_{tr}(Z,\mathbb{Q}_l)$ (the transcendental part) is the orthogonal complement of $H^{2k'}_{alg}(Z,\mathbb{Q}_l)$ under the intersection product. In other words:
\begin{eqnarray*}
H^{2k'}_{tr}(Z,\mathbb{Q}_l):=\{\alpha \in H^{2k'}(Z,\mathbb{Q}_l):~\alpha .\beta =0, ~\forall \beta \in H^{2k'}_{alg}(Z,\mathbb{Q}_l)\}. 
\end{eqnarray*} 
This decomposition is based on the following lemma.
\begin{lemma}
Assume that condition NH(Z) holds. Then the intersection product $H^{2k'}_{alg}(Z,\mathbb{Q}_l)\times H^{2k'}_{alg}(Z,\mathbb{Q}_l)\rightarrow \mathbb{Q}_l$ is non-degenerate and symmetric. 
\label{Lemma1}\end{lemma}
\begin{proof}
The symmetricity is clear because here $2k'$ is an even number.  Thus, only the non-degeneracy needs to be proven. 

Let $m=$ the dimension of $H^{2k'}(Z,\mathbb{Q}_l)$, and $x_1,\ldots ,x_m$ algebraic cycles (with coefficients in $\mathbb{Q}$) so that their image in $H^{2k'}(Z,\mathbb{Q}_l)$ generate $H^{2k'}_{alg}(Z,\mathbb{Q}_l)$. 
Let $\mathcal{Z}^{k'}_{\mathbb{Q}}(Z)$ be the group of algebraic cycles (with coefficients in $\mathbb{Q}$) modulo the numerical equivalence. By the next claim, we can find algebraic cycles $y_1,\ldots ,y_m$ (with coefficients in $\mathbb{Q}_l$) which satisfy the condition $x_i.y_j=\delta _{i,j}$ for all $i,j=1,\ldots ,m$. In particular, the images of $y_1,\ldots ,y_m$ in $H^{2k'}(Z,\mathbb{Q}_l)$ are linearly independent. In particular, they also form a basis for $H^{2k'}_{alg}(Z,\mathbb{Q}_l)$. 

{\bf Claim.} Assume $NH(Z)$ holds. Let $x_1,\ldots ,x_{m}$ be algebraic cycles (with rational coefficients) of codimension $k'$ on $Z$ so that their images in $H^{2k'}(Z,\mathbb{Q}_l)$ are linearly independent. Then, there are algebraic cycles $y_1,\ldots , y_m$ so that $x_i.y_j=\delta _{i,j}$ for all $i,j=1,\ldots ,m$. 

Proof of the Claim: 

We prove by induction on $m$. 

When $m=1$, since the image of $x_1$ in $H^{2k'}(Z,\mathbb{Q}_l)$ is not $0$, it follows from the $NH(Z)$ condition that the image of $x_1$ in $\mathcal{Z}^{k'}_{\mathbb{Q}}(Z)$ is not $0$. Therefore, by the definition of $\mathcal{Z}^{k'}_{\mathbb{Q}}(Z)$, there is an algebraic cycle $y_1$ (with rational coefficients) so that $x_1.y_1=1$. Hence the case $m=1$ is proven.   

We assume that the Claim holds for $m$, and will show that it also holds for $m+1$. Now, let $x_1,\ldots ,x_{m+1}$ be algebraic cycles (with rational coefficients) so that their images in $H^{2k'}(Z,\mathbb{Q}_l)$ are linearly independent. To prove the claim, it suffices to prove the following: There is an algebraic cycle (with rational coefficients) $y_{m+1}$ so that $x_{m+1}.y_{m+1}\not= 0$, but $x_1.y_{m+1}=\ldots =x_m.y_{m+1}=0$. Assume otherwise, then it follows that there must be constants $a_1,\ldots ,a_{m}\in \mathbb{Q}_l$ so that we have
\begin{equation}
x_{m+1}.y=a_1(x_1.y)+\ldots +a_m(x_m.y),\label{Equation6}
\end{equation} 
for all algebraic cycles $y$. The inductional assumption, applied to the $m$ elements $x_1,\ldots ,x_m$, implies the existence of an algebraic cycle $y_0$ with rational coefficients so that $x_1.y_0=1$ but $x_2.y_0=\ldots =x_m.y_0=0$. Choosing $y=y_0$ in (\ref{Equation6}), we obtain that $x_{m+1}.y_0=a_1$, hence $a_1\in \mathbb{Q}$. Then equation (\ref{Equation6}) becomes:
\begin{eqnarray*}
x_1'.y=a_2(x_2.y)+\ldots +a_m(x_m.y),
\end{eqnarray*}
for all algebraic cycles $y$, where $x_1'=x_{m+1}-a_1x_1$, $x_2,\ldots ,x_m$ are algebraic cycles with rational coefficients whose images in $H^{2k'}(Z,\mathbb{Q}_l)$ are linearly independent. The inductional assumption then implies that this cannot be the case. Thus, the case $m+1$ is proven. Hence, the claim is now proven (Q.E.D.). 

We continue the proof of the lemma. 

Now to show that the intersection product $H^{2k'}_{alg}(Z,\mathbb{Q}_l)\times H^{2k'}_{alg}(Z,\mathbb{Q}_l)\rightarrow \mathbb{Q}_l$ is non-degenerate, because the vector spaces involved are finite dimensional, it suffices to show that for any $0\not= \alpha \in H^{2k'}_{alg}(X,\mathbb{Q}_l)$, there is a $\beta \in H^{2k'}_{alg}(X,\mathbb{Q}_l)$ so that $\alpha .\beta \not= 0$. Since $x_1,\ldots ,x_m$ is a basis for $H^{2k'}_{alg}(X,\mathbb{Q}_l)$, there are $a_1,\ldots ,a_m\in \mathbb{Q}_l$ so that: 
\begin{eqnarray*}
\alpha =\sum _{i=1}^ma_ix_i,
\end{eqnarray*}
and because $\alpha \not= 0$, at least one of them, say $a_1$, is non-zero. Then 
\begin{eqnarray*}
\alpha .y_1=a_1\not= 0,
\end{eqnarray*}
hence the choice $\beta =y_1$ satisfies the requirement. 
\end{proof}

We let $\alpha _1,\ldots , \alpha _m$ be an orthogonal basis for $H^{2k'}_{alg}(X,\mathbb{Q}_l)$, with respect to the cup product (which always exist, since the characteristic of $\mathbb{Q}_l$ is $0$, and the cup product is symmetric). The non-degeneracy of cup product (Lemma \ref{Lemma1}) implies that $\alpha _i.\alpha _i\not= 0$ for all $i$.  If $x\in H^{2k'}(Z,\mathbb{Q}_l)$, we define
\begin{eqnarray*}
x'&=&\sum _{i=1}^m\frac{x.\alpha _i}{\alpha _i.\alpha _i}\alpha _i,\\
x"&=&x-x'.
\end{eqnarray*}
Then it is easy to check that $x'\in H^{2k'}_{alg}(Z,\mathbb{Q}_l)$, $x"\in H^{2k'}_{tr}(Z,\mathbb{Q}_l)$ and $x=x'+x"$. Moreover, this decomposition of $x$ is unique. Hence, we have the desired decomposition. We denote by $\tau :H^{2k'}(Z,\mathbb{Q}_l)\rightarrow H^{2k'}_{alg}(Z,\mathbb{Q}_l)$ the projection to the algebraic part. 

We also present another preliminary result before the proofs of the main results. Assume that $f:X\rightarrow X$ be a correspondence. Let $\alpha _1,\ldots ,\alpha _m$ be a basis for $H^{i}(X,\mathbb{Q}_l)$ and $\beta _1,\ldots ,\beta _m$ be a basis for $H^{2\dim (X)-i}(X,\mathbb{Q}_l)$. Fix arbitrary norms on $H^i(X,\mathbb{Q}_l)$ and $H^{2\dim (X)-i}(X,\mathbb{Q}_l)$, and an embedding of $\mathbb{Q}_l$ into $\mathbb{C}$. We let $|.|$ be the induced absolute value on $\mathbb{Q}_l$. Then there are positive constants $C_1,C_2>0$, independent of $f$, such that
\begin{equation}
C_1\sum _{p,q=1,\ldots ,m}|f^*(\alpha _p).\beta _q|\leq \|f^*|_{H^{i}(X,\mathbb{Q}_l)}\|\leq C_2\sum _{p,q=1,\ldots ,m}|f^*(\alpha _p).\beta _q|.
\label{Equation3}\end{equation} 
In fact, since the intersection product $$H^{i}(X,\mathbb{Q}_l)\times H^{2\dim (X)-i}(X,\mathbb{Q}_l)\rightarrow \mathbb{Q}_l$$ is non-degenerate, the following is a norm on $H^{i}(X,\mathbb{Q}_l)$: if $\alpha \in H^{i}(X,\mathbb{Q}_l)$ then
\begin{eqnarray*}
\|\alpha \|:=\sum _{q=1}^m|\alpha .\beta _q|.
\end{eqnarray*}
By definition, we then have
\begin{eqnarray*}
\|f^*|_{H^{i}(X,\mathbb{Q}_l)}\|:=\sup _{\|\alpha \|=1}\|f^*(\alpha )\|=\sup _{\|\alpha \|=1}\sum _{q=1}^m  |f^*(\alpha ).\beta _q|. 
\end{eqnarray*}
The left hand side inequality of (\ref{Equation3}) is obvious if we choose 
\begin{eqnarray*}
\frac{1}{C_1}=\sum _{p=1}^m\|\alpha _p\|.
\end{eqnarray*}

The right hand side inequality of (\ref{Equation3}) follows provided that if $\alpha =\sum _{p=1,\ldots ,m}x_p\alpha _p$ and $\|\alpha \|\leq 1$, then $\max _{p=1,\ldots m}|x_p|\leq C$ for some positive constant $C$. The latter claim is a simple consequence of the fact that on a finite dimensional vector space, any two norms are equivalent. We then apply this fact to two norms. The first is the norm $\|.\|$ chosen above. The second is the one $$\|\alpha \|':=\sum _{p=1}^m|x_p|.$$

Now we are ready for the proofs of the results. 

\begin{proof}[Proof of Theorem \ref{Theorem1}]  
1) Fix an integer $i$ between $0,\ldots ,2\dim (X)$. It is sufficient to prove that given $\lambda >\max _{i=0,\ldots ,\dim (X)}\lambda _i(f)$, then
\begin{eqnarray*}
\lim _{n\rightarrow\infty}\|(f^n)^*|_{H^i(X,\mathbb{Q}_l)}\|/\lambda ^n=0.
\end{eqnarray*} 

Let $\alpha _1,\ldots ,\alpha _m$ be a basis for $H^{i}(X,\mathbb{Q}_l)$ and $\beta _1,\ldots ,\beta _m$ be a basis for $H^{i}(X,\mathbb{Q}_l)$. Then by the observation at the beginning of this section, we have 
\begin{eqnarray*}
\|(f^n)^*|_{H^{i}(X,\mathbb{Q}_l)}\|\leq C\sum _{p,q=1,\ldots ,m}|(f^n)^*(\alpha _p).\beta _q|,
\end{eqnarray*}  
where $C$ is independent of $n$.

Hence, it is enough to show that for any $p,q=1,\ldots ,m$
\begin{eqnarray*}
\lim _{n\rightarrow\infty}|(f^n)^*(\alpha _p).\beta _q|/\lambda ^n=0.
\end{eqnarray*}

Let $pr_1,pr_2:X\times X\rightarrow X\times X$ be the two projections. Then, under the assumption that $NH(X\times X)$ holds, we have 
\begin{eqnarray*}
|(f^n)^*(\alpha _p).\beta _q|=|\Gamma _{f^n}.pr _2^*(\alpha _p).pr _1^*(\beta _q)|= |\Gamma _{f^n}.\tau (pr _2^*(\alpha _p).pr _1^*(\beta _q))|. 
\end{eqnarray*}

Since $\tau (pr _2^*(\alpha _p).pr _1^*(\beta _q))$ is represented by an algebraic cycle of dimension $k$, it then follows from the results in \cite{truong9} that we have the desired result
\begin{eqnarray*}
\lim _{n\rightarrow\infty}|(f^n)^*(\alpha _p).\beta _q|/\lambda ^n=0.
\end{eqnarray*}
More precisely, the results we used here are the following. First (Lemma 2.2 in \cite{truong9}), for any effective algebraic cycle $V$ of codimension $k$ on $X\times X$,   
$$|\Gamma _{f^n}.V|\leq C \deg (\Gamma _{f^n})\deg (V),$$
where $\deg (.)$ is the degree of an algebraic cycle in a fixed embedding of $X\times X$ into a projective space, and $C>0$ is a positive constant independent of $n$, $f$ and $V$.

Second (see (\ref{Equation1})),
\begin{eqnarray*}
\lim _{n\rightarrow\infty} \deg (\Gamma _{f^n})^{1/n}=\max _{i=0,\ldots ,\dim (X)}\lambda _i(f). 
\end{eqnarray*} 

2) The proof is similar, but here we need to use the Weil's Riemann hypothesis. Let $\lambda >\max _{i=0,\ldots ,\dim (X)}\lambda _i(f)$.  Since $f$ is a regular morphism, we obtain $\chi _{i}(f)=sp (f^*|_{H^i(X,\mathbb{Q}_l)})$, where $sp (f^*|_{H^i(X,\mathbb{Q}_l)})$ is the largest absolute value of the eigenvalues of $f^*:H^i(X,\mathbb{Q}_l)\rightarrow H^i(X,\mathbb{Q}_l)$. (Here again the absolute value is induced from the given embedding of $\mathbb{Q}_l$ into $\mathbb{C}$.) 

It suffices to consider the case where $X$ has positive characteristic $p$. We may assume that $X$ and $f$ are defined on some finite field $\mathbb{F}_q$. We recall briefly this well-known argument. There is (by collecting the coefficients in the defining equations for $X$ and $f$) a large subring $R$ of $\mathbb{K}$, finitely generated over $\mathbb{F}_p$, so that $X$ is a smooth fibre of a scheme $\mathcal{X}$ on $Spec(R)$ and $f$ is the generic fibre of a morphism $F:\mathcal{X}\rightarrow \mathcal{X}$ over $Spec(R)$. Let $X_0$ be a special fibre defined over a finite field $\mathbb{F}_q$, and $f_0=F|_{X_0}$. Define $\widetilde{X_0}$ and $\widetilde{f_0}$ to be the lifts of $X_0$ and $f_0$ to the algebraic closure $\overline{\mathbb{F}_p}$ of $\mathbb{F}_q$. The smooth base change theorem (Chapter 25 in \cite{milne}) then implies that
\begin{eqnarray*}
Tr[(f^n)^*:H^i(X,\mathbb{Q}_l)\rightarrow H^{i}(X,\mathbb{Q}_l)]=Tr[(\widetilde{f_0}^n)^*:H^i(\widetilde{X_0},\mathbb{Q}_l)\rightarrow H^i(\widetilde{X_0},\mathbb{Q}_l)],
\end{eqnarray*}
 for all $n$. While $\widetilde{X_0}$ may have more algebraic cycles than $X$, (\ref{Equation1}) implies that geometric dynamical degrees are lower-semicontinuous, and hence $\lambda _i(f)\geq \lambda _i(\widetilde{f_0})$ for all $i$. Therefore, if we can prove the conclusion for $\widetilde{X_0}$, then the conclusion for $X$ follows. Thus from now on we assume that $X$ is defined on a finite field.  

Let $\widetilde{Fr}:X\times X\rightarrow X\times X$ be the map $(x,y)\mapsto (x,Fr(y))$, where $Fr:X\rightarrow X$ is the Frobenius map. As a consequence of Deligne's proof of Weil's Riemann hypothesis, there is a polynomial $p_i(\widetilde{Fr})$  so that we have the generalised Lefschetz Trace Formula:
\begin{equation}
\Gamma _{f^n}.[p_i(\widetilde{Fr})^*\Delta ]=(-1)^iTr[(f^n)^*:H^i(X,\mathbb{Q}_l)\rightarrow H^i(X,\mathbb{Q}_l)].
\label{Equation4}\end{equation}

To see this, we use that each projection $p_i:~H^*(X,\mathbb{Q}_l)\rightarrow H^i(X,\mathbb{Q}_l)$ can be represented as a polynomial $p_i(Fr^*)$ in the pullback $Fr^*$ of the Frobenius map (part 1 of Theorem 2 in \cite{katz-messing}). [For the convenience of the readers, we reproduce their arguments here. In fact, the sets of eigenvalues of $Fr^*$ on two distinct cohomology groups $H^i(X,\mathbb{Q}_l)$ and $H^j(X,\mathbb{Q}_l)$ are disjoint, by the Weil's Riemann hypothesis. Hence if $R(t)$ is the characteristic polynomial of $Fr^*:H^*(X,\mathbb{Q}_l)\rightarrow H^*(X,\mathbb{Q}_l)$, $R_i(t)$ is the characteristic polynomial of $Fr^*:H^i(X,\mathbb{Q}_l)\rightarrow H^i(X,\mathbb{Q}_l)$ and $B_i(t)=R(t)/R_i(t)$ then $B_i(t)$ has no common zero with $R_i(t)$. Therefore, $B_i(Fr^*)$ is $0$ on $\oplus _{j\not= i}H^j(X,\mathbb{Q}_l)$ (by the Cayley - Hamilton theorem), and is invertible on $H^i(X,\mathbb{Q}_l)$. Then, by Cayley - Hamilton theorem again,  there is a polynomial $h(t)$ such that $h(0)=0$ and $h(B_i(Fr^*))|_{H^i(X,\mathbb{Q}_l)}=Id$. Since $B_i(Fr^*)$ is $0$ when restricted to $H^j(X,\mathbb{Q}_l)$ ($j\not= i$) and $h(0)=0$, it follows that $p_i(Fr^*):=h(B_i(Fr^*))$ is the projection $p_i:H^*(X,\mathbb{Q}_l)\rightarrow H^i(X,\mathbb{Q}_l)$.]

Note that $p_i(Fr^*)=p_i(Fr)^*$. Then, the Lefschetz Trace Formula (see e.g. Theorem 2.1 in \cite{milne1}) can be applied to the {\bf cohomological} correspondence $ (f^n)^*\circ p_i(Fr)^*=(p_i(Fr)\circ f^n)^*$:
\begin{eqnarray*}
[p_i(\widetilde{Fr})_*\Gamma _{f^n})].\Delta =(\Gamma _{p_i(Fr)\circ f^n}).\Delta = (-1)^iTr[(f^n)^*:H^i(X,\mathbb{Q}_l)\rightarrow H^i(X,\mathbb{Q}_l)].
\end{eqnarray*}
Then (\ref{Equation4}) is obtained by observing that by the projection formula: 
$$[p_i(\widetilde{Fr})_*\Gamma _{f^n})].\Delta  =\Gamma _{f^n}.[p_i(\widetilde{Fr})^*\Delta ],$$

Since the class of $p_i(\widetilde{Fr})^*\Delta $ is an algebraic cycle (with rational coefficients), the proof is completed by observing that similarly to part 1)
\begin{eqnarray*}
\limsup _{n\rightarrow\infty}|\Gamma _{f^n}.[p_i(\widetilde{Fr})^*\Delta ]|/\lambda ^n=0,
\end{eqnarray*} 
and that (using $(f^n)^*=(f^*)^n$ for a regular morphism $f$)
\begin{eqnarray*}
\limsup _{n\rightarrow\infty}|Tr[(f^n)^*:H^i(X,\mathbb{Q}_l)\rightarrow H^i(X,\mathbb{Q}_l)]|^{1/n}=sp(f^*|_{H^{i}(X,\mathbb{Q}_l)}). 
\end{eqnarray*}
The last (elementary) equality can be deduced from the following simple claim, which we leave to the readers to verify. 

Claim: Let $\mu _1,\ldots ,\mu _m$ be complex numbers with $|\mu _1|=\ldots =|\mu _m|=1$. For any $\epsilon >0$, there exist infinitely many values of positive integers $k$ such that $|\mu _i^k-1|<\epsilon $, and in particular $Re(\mu _j^k)>1-\epsilon $ for all $j$.  

\end{proof}
\begin{remarks}
In the original proof of part 2) above, we used the usual Lefschetz Trace Formula. Then similarly we can bound the alternative sum 
$$\sum _{i=0}^{2\dim (X)}(-1)^iTr[(f^n)^*:H^i(X,\mathbb{Q}_l)\rightarrow H^i(X,\mathbb{Q}_l)],$$
in terms of the geometric dynamical degrees $\lambda _i(f)$. However, there may be some cancelations in the alternative sum of the traces which do not quite give us the inequality we need. We thank Peter O'Sullivan for pointing this out and for suggesting the correction which we used here. 
\label{Remark2}\end{remarks}

There is a subtlety when applying the argument of reduction to finite fields in the proof of part 2) of Theorem \ref{Theorem1} to iterates of correspondences (for example, the finite fields may increase when we increase the iterates). If $X$ is already defined on a finite field, then such a reduction is not needed. There is still a difficulty arising from the fact that in general we do not have $(f^n)^*=(f^*)^n$, and hence the eigenvalues of $(f^n)^*$ may not be related to those of $f^*$. However, this can be dealt with by a modification of the proof, and this gives us a proof of Theorem \ref{Theorem6}. Below we provide a detailed argument. 
\begin{proof}[Proof of Theorem \ref{Theorem6}]
Since $\mathbb{K}=\overline{\mathbb{F}_p}$ is the closure of a finite field, $X$ is actually defined on a finite field $\mathbb{F}_q$, where $q$ is a power of $p$. Then we have (see the proof of part 2) of Theorem \ref{Theorem1}) that the projections $p_i:H^*(X,\mathbb{Q}_l)\rightarrow H^i(X,\mathbb{Q}_l)$ are all algebraic, that is $p_i=p_i(Fr^*)$ for some polynomials in the pullback $Fr^*$ of the Frobenius $Fr:X\rightarrow X$. We let $\widetilde{Fr}:X\times X\rightarrow X\times X$ be the map $(x,y)\mapsto (x,F(y))$. Then, by the proof of the Lefschetz trace formula (see e.g. Theorem 2.1 in \cite{milne1}), for any {\bf generalised} correspondence (with integer coefficients, rather than only positive coefficients) $\phi :X\rightarrow X$, we have
\begin{eqnarray*}
\Gamma _{\phi}. p_i(\widetilde{Fr})^*(\Delta )=(-1)^iTr[\phi :H^{i}(X,\mathbb{Q}_l)\rightarrow H^{i}(X,\mathbb{Q}_l)]. 
\end{eqnarray*}  

We will use this to prove the following claim:

Claim. Let $f:X\rightarrow X$ be a dominant correspondence. Then, for all $i=0,\ldots ,2\dim (X)$
\begin{eqnarray*}
sp (f^*|_{H^{i}(X,\mathbb{Q}_l)})\leq C \deg (\Gamma _f),
\end{eqnarray*}
 where $C>0$ is independent of $f$. Here, as in the case of regular morphisms, $sp (f^*|_{H^i(X,\mathbb{Q}_l)})$ is the spectral radius of the pullback $f^*:H^{i}(X,\mathbb{Q}_l)\rightarrow H^{i}(X,\mathbb{Q}_l)$. 
 
 Proof of Claim. For any $n$, we consider the {\bf cohomological} correspondence $\phi _n:=(f^*)^n:H^*(X,\mathbb{Q}_l)\rightarrow H^*(X,\mathbb{Q}_l)$. Since $\phi _1=f^*$ is algebraic, it follows that all $\phi _n$ are algebraic. That is, we can write $\phi _n=(f_n^+)^*-(f_n^-)^{*}$, where $f_n^{\pm}$ are effective algebraic cycles on $X\times X$. Moreover, an iterated use of Lemma 2.2 in \cite{truong9} shows that we can arrange to have the estimates
 \begin{eqnarray*}
 \deg (f_n^{\pm})\leq (2C)^n\deg (\Gamma _f)^n,
 \end{eqnarray*}
 for all $n$. Here $C>0$ is the constant in Lemma 2.2 in \cite{truong9}. It follows again from this Lemma that 
 \begin{eqnarray*}
 |Tr[(f^*)^n:H^{i}(X,\mathbb{Q}_l)\rightarrow H^i(X,\mathbb{Q}_l)]|&=&|Tr[(f^*)^n:H^{i}(X,\mathbb{Q}_l)\rightarrow H^i(X,\mathbb{Q}_l)]|\\
 &=&|(f_n^+-f_n^-).p_i(\widetilde{Fr})^*(\Delta )|\\
 &\leq&C(\deg (f_n^+)+\deg (f_n^-))\\
 &\leq&C(2C)^n\deg (\Gamma _f)^n.  
 \end{eqnarray*}
 Therefore, 
 \begin{eqnarray*}
 2C\deg (\Gamma _f)\geq \limsup _{n\rightarrow\infty} |Tr[(f^*)^n:H^{i}(X,\mathbb{Q}_l)\rightarrow H^i(X,\mathbb{Q}_l)]|^{1/n}=sp(f^*|_{H^{i}(X,\mathbb{Q}_l)}). 
 \end{eqnarray*}
 Here the constant $2C>0$ is independent of $f$. (Q.E.D.)
 
 Now we continue the proof of the theorem. Applying the Claim to iterates $f^n$ and using (\ref{Equation1}), we have
 \begin{eqnarray*}
\chi _i(f)&=&\limsup _{n\rightarrow\infty} sp((f^n)^*|_{H^{i}(X,\mathbb{Q}_l)})^{1/n}\\
&\leq& \limsup _{n\rightarrow\infty}[C\deg (\Gamma _{f^n})]^{1/n}\\
&=&\max \{\lambda _0(f),\ldots ,\lambda _p(f)\}. 
 \end{eqnarray*}

\end{proof}

\begin{proof}[Proof of Theorem \ref{Theorem2}]
1) Since we assume that $NH(X\times X)$ holds,  part 1) of Theorem \ref{Theorem1} applies. Hence, we have $\max \{\chi _1(f),\chi _2(f),\chi _3(f)\}\leq \max \{\lambda _0(f),\lambda _1(f),\lambda _2(f)\}$. The right hand side in the above inequality is $\lambda _1(f)$ by the other assumption in the theorem. From the obvious inequality $\chi _2(f)\geq \lambda _1(f)$,  we obtain the conclusion of the theorem. 

1') Since we assume that $\mathbb{K}=\overline{\mathbb{F}_q}$, Theorem \ref{Theorem6} applies. Then we argue similarly to part 1). 

2) Since we assume that $f$ is a regular morphism, part 2) of Theorem \ref{Theorem1} applies. Then we argue similarly to part 1). 

\end{proof}

\section{Dynamical degrees and the Weil's Riemann hypothesis}
\label{Section3}
{\bf Convention.} As in Section \ref{Section2}, for simplicity we suppress all the Tate twists in the l-adic cohomology groups.

For the convenience of the readers, we first recall some backgrounds about the Weil's Riemann hypothesis. Let $X_0$ be a smooth projective variety defined over a finite field $\mathbb{F}_q$. Let $X$ be the lift of $X_0$ to an algebraic closure $\overline{\mathbb{F}_q}$. Let $Fr_X:X\rightarrow X$ be the Frobenius morphism. A simple expression of it is as follows. On a projective space $\mathbb{P}^N$, $Fr[x_0:\ldots :x_N]=[x_0^q:\ldots :x_N^q]$. If $X\subset \mathbb{P}^N$, then $Fr_X$ is simply the restriction of $Fr$ to $X$. The Weil's Riemann hypothesis is then the following statement. It was solved by Piere Deligne in the $1970$'s.

{\bf Weil's Riemman hypothesis.} If $\alpha$ is an eigenvalue of $Fr^*:H^i(X,\mathbb{Q}_l)\rightarrow H^i(X,\mathbb{Q}_l)$, then $|\alpha |=q^{i/2}$.

Here are some preliminary reductions of Weil's Riemann hypothesis (Chapter 28 in \cite{milne}). It is enough to solve the conjecture for any finite extension of $\mathbb{F}_q$. Another reduction is that, in the statement of the conjecture, it is enough to  show that $|\alpha |\leq q^{i/2}$. In \cite{milne}, the second reduction was proven by showing that if $\alpha$ is an eigenvalue of $Fr^*$ on $H^{i}(X,\mathbb{Q}_l)$, then $q^{\dim (X)}/\alpha$ is an eigenvalue of $Fr^*$ on $H^{2\dim (X)-i}(X,\mathbb{Q}_l)$. Here is another simple proof of this reduction, using Poincar\'e duality only. 
\begin{proof}[Proof of the second reduction]
 Assume that $0\not= v\in H^{i}(X,\mathbb{Q}_l)$ is an eigenvector of $Fr^*$ with eigenvalue $\alpha$. Then $|\alpha |\leq q^{i/2}$ by the hypothesis in the reduction. We will show that $|\alpha |=q^{i/2}$. To this end, let $w\in H^{2\dim (X)-i}(X,\mathbb{Q}_l)$ be so that $v.w=1$ (such a $w$ always exists since the intersection product $H^i(X,\mathbb{Q}_l)\times H^{2\dim (X)-i}(X,\mathbb{Q}_l)\rightarrow \mathbb{Q}_l$ is non-degenerate). Since the topological degree of $Fr$ is $q^{\dim (X)}$ (see Lemma 27.1 in \cite{milne}, or for another proof using the product formula for dynamical degrees $\lambda _i$'s see the proof of Claim 2 below), we deduce that 
 \begin{eqnarray*}
 q^{n\dim (X)}=(Fr^n)^*(v.w)=(Fr^n)^*(v).(Fr^n)^*(w)=\alpha ^{ni/2}v.(Fr^n)^*(w). 
 \end{eqnarray*}
  Since by the hypothesis of the reduction, the absolute values of the eigenvalues of $Fr^*$ on $H^{2\dim (X)-i}(X,\mathbb{Q}_l)$ are all $\leq q^{\dim (X)-i/2}$, the growth of $(Fr^n)^*(w)$ is bounded by $n^mq^{n\dim (X)-ni/2}$, for some constant $m$. It follows that $|\alpha |=q^{i/2}$ as wanted.   
\end{proof}
In terms of the cohomological dynamical degrees, the second reduction is equivalent to the statement that $\chi _i(Fr)\leq q^{i/2}$ for all $i=0,\ldots ,2\dim (X)$. By another elementary reduction (using product of spaces and maps), we obtain the following

{\bf Reduction.} Weil's Riemann hypothesis is equivalent to the statement that $\chi _{2i}(Fr)\leq q^{i}$ for all $i=0,\ldots ,\dim (X)$.  
 
 The proof of Theorem \ref{Theorem4} follows from the following claims. 

{\bf Claim 1.} Assume that we have the expected equality $\lambda _i(f)=\chi _{2i}(f)$ holds, for all smooth projective varieties $X$ and regular morphisms $f$ on $X$, and for all $i=0,\ldots ,\dim (X)$. Then the Weil's Riemann hypothesis holds. 
\begin{proof}
Applying the expected equality to iterates of the Frobenius map $Fr$, we find that $\chi _{2i}(Fr)=\lambda _i(Fr)$. The $\lambda _i(Fr)$ is easy to compute, and in this case is $q^{i}$. (For a fantasy proof of this fact, we can consider a dominant regular morphism $\pi :X\rightarrow \mathbb{P}^k$, and apply the product formula in \cite{truong9, truong}. See the proof of Claim 2 below for more details.) By definition of $\chi _{2i}(Fr)$, any eigenvalue of $Fr$ on $H^{2i}(X,\mathbb{Q}_l)$ has absolute value $\leq q^i$. By the preliminary reductions mentioned above, this is enough to prove the Weil's Riemann hypothesis.   
\end{proof}
 By Claim 1, the expected equality $\lambda _i(f)=\chi _{2i}(f)$ (which as noted before, holds in the case $\mathbb{K}=\mathbb{C}$) is a generalisation of the Weil's Riemann hypothesis.  

{\bf Claim 2.} Assume that the product formula holds for the cohomological dynamical degrees $\chi _{2i}$, and where $f$ and $g$ are both regular morphisms semi-conjugated by a dominant regular morphism with finite fibres $\pi :X\rightarrow \mathbb{P}^k$. Then the Weil's Riemann hypothesis holds. 
\begin{proof}
Let $\dim (X)=k$. There is always a dominant regular morphism $\pi :X\rightarrow Y=\mathbb{P}^k$ with finite fibres and which is defined on $\mathbb{F}_q$, for example by using Noether's normalisation theorem. The Frobenius maps have the important property that the equality $\pi \circ Fr_X=Fr_Y\circ \pi$ is always satisfied. Since $\dim (X)=k=\dim \mathbb{P}^k$, by the assumptions in Claim 2, we have
\begin{equation}
\chi _{2i}(Fr_X)=\chi _{2i}(Fr_Y)\chi _0(Fr_X|\pi )=q^i.  
\label{Equation2}\end{equation}
This is the conclusion of the Weil's Riemann hypothesis. Here we used that the Weil's Riemann hypothesis is true for $Y=\mathbb{P}^k$ (because the cohomology group of $\mathbb{P}^k$ is very simple and is generated by algebraic cycles) and $\chi _0(Fr_X|\pi )=1$ (following from $\chi _0(Fr_X)=\chi _0(Fr_Y)\chi _0(Fr_X|\pi )$ and $\chi _0(Fr_X)=\chi _0(Fr_Y)=1$). 
 \end{proof}
By Claim 2, the product formula for cohomological dynamical degrees  is also a generalisation of the Weil's Riemann hypothesis. 

{\bf Claim 3.} Assume that we have Dinh's inequality $\chi _i(f)^2\leq \max _{j+l=i}\lambda _j(f)\lambda _l(f)$ for all $i$ and regular morphisms in positive characteristic. Then the Weil's Riemann hypothesis holds. 
\begin{proof}
The proof is similar to those of the above claims, by observing that applying Dinh's inequality to the Frobenius map gives the desired inequality
\begin{eqnarray*}
\chi _i(Fr)^2\leq \max _{j+l=i}\lambda _j(Fr)\lambda _l(Fr)=\max _{j+l=i}q^jq^l=q^i.
\end{eqnarray*}
\end{proof}
By Claim 3, Dinh's inequality for cohomological dynamical degrees is yet another generalisation of the Weil's Riemann hypothesis.  

\section{An approach to Questions 2 and 3}
\label{Section4}

To Questions 2 and 3 in general, we propose to study the following two conditions. 
 
 {\bf Condition (A).} Let $X$ be a smooth projective variety over $\mathbb{K}$. Let $f_1,f_2:X\rightarrow X$ be two correspondences. Assume that $f_1\geq f_2$, that is there is an effective algebraic cycle $\Gamma  $ on $X\times X$ so that $\Gamma _{f_1}=\Gamma _{f_2}+\Gamma $. Then there is a positive constant $C>0$, independent of the correspondences $f_1$ and $f_2$, such that $C\|f_1^*|_{H^{2i}(X,\mathbb{Q}_l)}\|\geq \|f_2^*|_{H^{2i}(X,\mathbb{Q}_l)}\|$ for all $i=0,\ldots ,\dim (X)$. 
  
 {\bf Condition (B).} Let $X$ and $Y=\mathbb{P}^k$ be smooth projective varieties of the same dimension $k$. Let $\pi :X\rightarrow Y$ be a surjective regular morphism whose all fibres are {\bf finite}. Let $g:Y\rightarrow Y$ be a correspondence, and let $f:X\rightarrow X$ be the correspondence whose graph is $\Gamma _f:=(\pi \times \pi )^*(\Gamma _g)$. Then there exists a constant $C>0$, independent of $g$, so that
 \begin{eqnarray*}
 \|f^*|_{H^{2i}(X,\mathbb{Q}_l)}\|\leq C\|g^*|_{H^{2i}(Y,\mathbb{Q}_l)}\|,
 \end{eqnarray*}  
 for all $i=0,\ldots ,k$. Here we fix arbitrary norms on the finite-dimensional vector spaces $H^{2i}(X,\mathbb{Q}_l)$ and $H^{2i}(Y,\mathbb{Q}_l)$. 
 
 
 
 
 Some remarks are in order. 
 
 \begin{remarks}
 1) Condition (A) is satisfied if on $H^{2i}(X,\mathbb{Q}_l)$ we have a positivity notion of cohomology classes as in the case $\mathbb{K}=\mathbb{C}$. 
  
 2) For any projective variety $X$ of dimension $k$, there are always (by using generic projections from linear subspaces of projective spaces containing $X$) surjective regular morphisms with {\bf finite} fibres $\pi :X\rightarrow \mathbb{P}^k$.
 
 3) The correspondence $f$ defined in Condition (B) was called the pullback of $g$ by $f$ and was denoted as $\pi ^*(g)$ in \cite{truong9}. In this special case, the cohomological class of $\Gamma _f$ is exactly the same as the pullback under $\pi \times \pi$ of the cohomological class of $\Gamma _g$. For a general dominant regular morphism (more generally, a dominant rational map) with generically finite fibres $\pi :X\rightarrow Y$, we can still define the pullback $\pi ^*(g)$ of any correspondence $g:Y\rightarrow Y$. However, in the general case, no relation is expected for the cohomological classes of $g$ and $\pi ^*(g)$.    
 
 It can be checked that in Condition (B), $\pi \circ f^n=\deg (\pi )^ng^n\circ \pi$ for all $n$. It is shown in \cite{truong9} that we then have $\lambda _i(f)=\deg (\pi )\lambda _i(g)$ for all $i$. 
 
 \label{Remark1}\end{remarks}

The next result concerns Questions 2 and 3. 
 
 \begin{theorem}
 
 1) Condition (B) is always satisfied. 
  
 2) If Condition (A) is satisfied then Question 2 has an affirmative answer. 
  
 
 3) If Question 2 has an affirmative answer then in the definition of $\chi _{2i}(f)$ we can replace limsup by lim. 
 
 4) For rational maps, an affirmative answer for Question 2 implies an affirmative answer for Question 3. 
 \label{Theorem3}\end{theorem}
\begin{proof}[Proof of Theorem \ref{Theorem3}]
 We first make some preparations. Let $\pi :X\rightarrow Y=\mathbb{P}^k$ be a surjective regular morphism with {\bf finite} fibres. Let $f:X\rightarrow X$ be a correspondence. For any positive integer $n$, we define a correspondence $g_n:Y\rightarrow Y$ given by declaring $\Gamma _{g_n}=(\pi\times \pi)_*(\Gamma _{f^n})$. Note that even if $f$ is a regular morphism, $g_n$ will rarely be a regular morphism or even a rational map. Also, in general $f^n$ and $g_n$ are not semi-conjugate, even up to a multiplicative constant.  We overcome this by defining a correspondence $f_n:X\rightarrow X$ by declaring $\Gamma _{f_n}=(\pi\times \pi)^*(\Gamma _{g_n})$.  We note that the cohomology groups of $Y=\mathbb{P}^k$ are very simple, in particular generated by algebraic cycles: $H^{2i}(Y,\mathbb{Q}_l)=H^{2i}_{num}(Y,\mathbb{Q}_l)$. 

Here are some relations between $f^n, g_n$ and $f_n$. First we have $f^n\leq f_n$. Second, we have $\lambda _i(f_n)=\deg (\pi )\lambda _i(g_n)$ for all $n$ and $i$ (see \cite{truong9}). Last, we have
\begin{eqnarray*}
\|g_n^*|_{H^{2i}(Y,\mathbb{Q}_l)}\|\leq C\|(f^n)^*|_{H^{2i}(X,\mathbb{Q}_l)}\|, 
\end{eqnarray*} 
where $C>0$ is independent of $f$ and $n$. In fact, let $h$ be the class of a hyperplane in $Y=\mathbb{P}^k$. Then $H^{2i}(Y,\mathbb{Q}_l)$ is generated by $h^i$. Let $pr_1,pr_2$ denote either the projections $X\times X\rightarrow X$ or $Y\times Y\rightarrow Y$ (the meaning will be clear from the context). Then 
\begin{eqnarray*}
\|g_n^*|_{H^{2i}(Y,\mathbb{Q}_l)}\|&=&|g_n^*(h^i).h^{k-i}|=|\Gamma _{g_n}.pr _{2}^*(h^i).pr _1^*(h^{k-i})|\\
&=&|(\pi \times \pi)_*(\Gamma _{f^n}).pr _{2}^*(h^i).pr _1^*(h^{k-i})|\\
&=&|\Gamma _{f^n}.(\pi \times \pi)^*(pr _{2}^*(h^i).pr _1^*(h^{k-i}))|\\
&=&|\Gamma _{f^n}.pr _{2}^*\pi ^*(h^i).pr _1^*\pi ^*(h^{k-i}))|\\
&=&|(f^n)^*(\pi ^*(h^i)).\pi ^*(h^{k-i})|\\
&\leq& C\|(f^n)^*|_{H^{2i}(X,\mathbb{Q}_l)}\|.
 \end{eqnarray*} 

1) We first observe that if $\alpha \in H^{i}(X,\mathbb{Q}_l)$ is such that $\pi _*(\alpha )=0$ in $H^{i}(Y,\mathbb{Q}_l)$, then $f_n^*(\alpha )=(f_n)_*(\alpha )=0$. We show for example that $f_n^*(\alpha )=0$. To this end, it suffices to show that for any $\beta \in H^{2k-i}(X,\mathbb{Q}_l)$ then $f_n^*(\alpha ).\beta =0$. In fact, we have 
\begin{eqnarray*}
f_n^*(\alpha ).\beta &=&\Gamma _{f_n}.pr_2^*(\alpha ).pr_1^*(\beta )\\
&=&(\pi \times \pi )^*(\Gamma _{g_n}).(pr _2^*(\alpha ).pr _1^*(\beta ))\\
&=&\Gamma _{g_n}.(\pi\times \pi)_*(pr _2^*(\alpha ).pr _1^*(\beta ))\\
&=&\Gamma _{g_n}.pr _2^*(\pi _*\alpha ). pr _1^*(\pi _*\beta ).    
\end{eqnarray*}
The last number is $0$ provided $\pi _*(\alpha )=0$, as assumed. (Note that it is also $0$ if $\pi _*(\beta )=0$.) Here, we used that under the assumptions on $\pi$, the cohomology class of $\Gamma _{f_n}$ is the same as the cohomology class of $(\pi \times \pi)^*(\Gamma _{g_n})$. We also used that $(\pi\times \pi)_*(pr _2^*(\alpha ).pr _1^*(\beta ))=pr _2^*(\pi _*\alpha ). pr _1^*(\pi _*\beta )$.  (This can be seen very easily in the case $\alpha $ and $\beta$ are represented by irreducible subvarieties of $X$, since in this case $pr _2^*(\alpha ).pr _1^*(\beta )$ is represented by the variety $\beta \times \alpha $ in $X\times X$, whose image by $\pi\times \pi $ is exactly $\pi (\beta )\times \pi (\alpha )$. The general case can be proceeded similarly, by using the Kunneth's formula for the $l$-adic cohomology.) 

From the above observation and the decomposition $H^*(X,\mathbb{Q}_l)=\pi ^*(H^*(Y,\mathbb{Q}_l))\oplus~ Ker (\pi _*)$, it follows that 
\begin{eqnarray*}
\|f_n^*|_{H^{i}(X,\mathbb{Q}_l)}\|=\|f_n^*|_{\pi ^*(H^i(Y,\mathbb{Q}_l))}\|=\deg (\pi )^2\|g_n|_{H^i(Y,\mathbb{Q}_l)}\|,
\end{eqnarray*}
provided that the norm on $\pi ^*(H^i(Y,\mathbb{Q}_l)$ is induced from the norm on $H^i(Y,\mathbb{Q}_l)$. This completes the proof. 

2) Assume that Condition (A) is satisfied. Then, from $f_n\geq f^n$ for all $n$, we have
\begin{eqnarray*}
C\|(f_n)^*|_{H^{2i}(X,\mathbb{Q}_l)}\|\geq \|(f^n)^*|_{H^{2i}(X,\mathbb{Q}_l)}\|. 
\end{eqnarray*}
Hence, from the inequalities obtained above, we get
\begin{eqnarray*}
\|(f^n)^*|_{H^{2i}(X,\mathbb{Q}_l)}\|&\geq& |(f^n)^*(\pi ^*(h^i)).\pi ^*(h^{k-i})|= \|g_n^*|_{H^{2i}(Y,\mathbb{Q}_l)}\|\\
&=&\frac{1}{d^2}\|f_n^*|_{H^{2i}(X,\mathbb{Q}_l)}\|\geq \frac{1}{C} \|(f^n)^*|_{H^{2i}(X,\mathbb{Q}_l)}\|.
\end{eqnarray*}
By (\ref{Equation1}), we obtain
\begin{eqnarray*}
\limsup _{n\rightarrow\infty}|(f^n)^*(\pi ^*(h^i)).\pi ^*(h^{k-i})|^{1/n}\leq \lambda _i(f).
\end{eqnarray*}
The proof is thus completed. 

3) This easily follows from similar arguments. 

4) This follows from the results in \cite{truong9} for $\lambda _i$.

\end{proof}

\end{document}